\documentclass{amsart}
\usepackage{amssymb}
\usepackage{mathrsfs}
\usepackage{hyperref}
\usepackage{verbatim}
\usepackage{graphicx}
\usepackage{subfigure}
\usepackage{psfrag}
\usepackage{enumitem}

\newtheorem{theorem}{Theorem}[section]
\newtheorem{lemma}[theorem]{Lemma}
\newtheorem{corollary}[theorem]{Corollary}

{\theoremstyle{definition}

}

\begin{document}
\title[On mappings with Jacobian one]{On mappings with Jacobian one}

\author[Z. Jelonek]{Zbigniew Jelonek}

\address[Zbigniew Jelonek]{ Instytut Matematyczny, Polska Akademia Nauk, \'Sniadeckich 8, 00-656 Warszawa, Poland      
\newline  
              E-mail: {\tt najelone@cyf-kr.edu.pl}
              }

\subjclass[2010]{Primary: 14R15; Secondary: 14P99.}

\keywords{Jacobian conjecture, local polynomial diffeomorphism}

\date{\today}

\thanks{The  author was partially supported by  the grant of Narodowe Centrum Nauki, number 2024/55/B/ST1/01412 and  by the São Paulo Research
Foundation (FAPESP) under grant number 2025/06706-0.}

\maketitle

\begin{abstract}
We show that the set $A(n, d)$ of polynomial automorphisms $F : \Bbb C^n \to \Bbb C^n$ of degree at most  $d$ and with $Jac(F ) = 1$ is Zariski closed.
Let  $X(n,d)$ denotes the variety of all polynomial mappings of degree bounded by $d$ and with Jacobian $1.$ Then every component of this set  either consists with polynomial automorphisms, either generically it is composed wiith  counterexamples 
to the Jacobian Conjecture. In  particular
if the set $X(n,d)$ is irreducible, and $n\ge 3, d\ge 7$, then a generic element of this set is a counterexample to the Jacobian Conjecture.
\end{abstract}

\maketitle

\section{Introduction}
Nowadays there is a lot of interests in  counterexamples to the Jacobian Conjecture. Here we try to explain 
 how possible counterexamples are located in the variety $X(n,d)$ of all polynomial mappings of degree bounded by $d$ and with Jacobian $1.$ 

The results presented here has actually been known to me since 2009, when I uploaded a note titled "On Polynomial Automorphisms" to ResearchGate (January 2009), 
in which I proved the key Lemma 2.1. However, until now, I had not decided to publish it formally. Following the recent discovery of a counterexample to the Jacobian Conjecture, my colleagues encouraged me to publish it. 

It was suggested already by Dru\.zkowski and Rusek
 \cite{dr}, \cite{rus}, that the variety $X(n,d)$ of such polynomials mappings is not irreducible (in fact they have considered the special case of the variety of maps in the Dru\.zkowski
 form). Here we prove that every component of this set  either consists with polynomial automorphisms, either generically it is composed with  counterexamples 
to the Jacobian Conjecture. It does not mean that $X(n,d)$ is reducible, but if not, then a generic element of this set is a counterexample to the Jacobian Conjecture,
at least if $n\ge 3, d\ge 7.$

\vspace{5mm}

{\bf Acknowledgement}. The author is grateful to Professor Lev Birbrair for the helpful discussions. The  author was partially supported by  the grant of Narodowe Centrum Nauki, number 2024/55/B/ST1/01412.

\section{Main Result}

This Lemma was first published in 2009 on Research gate - see \cite{jel}.

\begin{lemma}
The set $A(n, d)$ of polynomial automorphisms $F : \Bbb C^n \to \Bbb C^n$ of degree at most  $d$ and with $Jac(F ) = 1$ is Zariski closed.
\end{lemma}

\begin{proof}
It is easy to see that the set $A(n, d)$ is a constructible subset of $\Omega(n,d)\cong\Bbb C^M$ (where  $\Omega(n,d)$ is a space of all  mappings $F:\Bbb C^n \to \Bbb C^n$ of degree at most $d$).  Indeed, consider the algebraic set 
$$S=\{ (F,G)\in \Omega(n,d)\times \Omega(n, d^{n-1}) ; F\circ G=identity\}.$$ Then the projection $\pi(S)$ of
$S$ on  $\Omega(n,d)$ is a constructible set by Chevaley theorem. But $A(n,d)=\pi(S) \cap \{ F : Jac(F)=1\},$ hence it is a constructible set.

Thus it is enough to prove that the set $A(n, d)$ is closed in the euclidian topology. Let $f_k \in A(n, d)$ and $f_k \to F.$ 
Hence $F = (F_1, ..., F_n)$ is a polynomial mapping with $Jac(F ) = 1.$ We can assume that $f_n (0) = 0.$ Indeed, $f_k (0) = b_k$ and $b_k \to b = F (0).$

 If we show that $f_k - b_k \to G \in A(n, d)$,  then also $f_k \to G + b.$ Let $f_k(1, ..., 1) = (w_{k,1} , ..., w_{k,n}) = w_k \in \Bbb C^n.$
  Hence $w_k \to w = F (1, ..., 1).$  If $g_k = f^{ -1}_k,$ then $g_k(w_k) = (1, ..., 1).$  Note that the family of all $w_k$ is bounded.
  Now let $g_k = (Q_{k,1} , ..., Q_{k,n}).$ It is well known that the algebraic degree of $g_k$ is uniformly bounded by $d^{n-1}.$
  
   For a polynomial $Q = \sum a_\alpha x^\alpha\in  \Bbb C[x_1, ..., x_n]$ let $||Q|| =max |a_\alpha|.$ For a non-zero polynomial $Q$ let us denote $\tilde{Q} = Q/||Q||.$
    Since $Q_{k,j} (w_k) \to 1$ and $w_k$ is bounded, we have that there is an $\epsilon > 0$, such that $||Q_{k,j} || > \epsilon$ for every $k, j.$
    
    We have $Q_{k,1} (f_{k,1} , ..., f_{k,n} ) = x_1$. Consequently $\tilde{Q}_{k,1} (f_{k,1} , ..., f_{k,n} ) = a(k) x_1$, where $0 <a(k) < 1/\epsilon$. 
    If we take a subsequence of $\tilde{Q}_{k,1}$ we can assume that it converges to a polynomial $Q_1$ and one of its (fixed) coefficients has norm $1$. 
    Moreover we can assume that $a(k)$ also converges to $a$. Let us note that $a  \not= 0,$ because otherwise $Q_1(F_1, ..., F_n) = 0$ and $Q_1$ is a non-zero polynomial. 
    
    This contradicts the fact that $F_1, ..., F_n$ are algebraically independent, as components of a mapping with Jacobian equal to $1.$
     Hence finally $(1/a)Q_1(F_1, ..., F_n) = x_1,$ i.e., there is a polynomial $Q'_1$ such that $Q'_1(F_1, ..., F_n) = x_1.$
      In the same way there are polynomials $Q'_2, ..., Q'_n$ such that $Q'_j (F_1, ..., F_n) = x_j$ for $j = 2, ...n.$ Hence $Q'\circ F=identity$. By the symmetry also $F\circ Q'=identity,$ hence $F$ is bijective and $Q'=F^{-1}$ (note that $Q'=lim \ Q'_{k_j},$ where $Q'_{k_j}\in A(n,d)$ and $Q'_{k_j}=f_{k_j}^{-1})$.
\end{proof}

\begin{theorem} Let $X(n,d)$ denotes the space (which is an affine algebraic variety) of polynomial mappings $F$ with $Jac(F)=1$ and of degree bounded by $d.$ Let $Y$ be an irreducible component of this set. Then either $Y$ consists of polynomial automorphisms, or a generic element of $Y$ is a counterexample to the Jacobian Conjecture. Moreover, $\dim Y \ge n^2-1.$
\end{theorem}

\begin{proof} Let $Y$ be an irreducible component. Then either $Y\subset A(n,d)$ or it contains a counterexample to the Jacobian Conjecture. In the latter case since $A(n,d)$ is closed it means that $A(n,d)\cap Y$ is a nowhere dense in $Y.$ Now denote by $\mathcal L$ the set of linear automorphisms with the Jacobian $1.$ Take $F\in Y$, which is not contained in any different component of $X(n,d).$
Then $F\circ \mathcal{L}$ is an irreducible set which is contained in $X(n,d).$ Since it has a common point which belongs to $Y$ and does not belong to any other component it means that $F\circ \mathcal{L}\subset Y.$ But $\dim F\circ \mathcal{L}=n^2-1.$
\end{proof}

\begin{corollary}
If the set $X(n,d)$ is irreducible, and $n\ge 3, d\ge 7$, then a generic element of this set is a counterexample to the Jacobian Conjecture. 
\end{corollary}

\begin{proof}
Indeed, due to \cite{cjc} we know that $X(n,d)$ is not composed only with polynomial automorphisms.
\end{proof}

\end{document}